\documentclass[12pt,bezier]{article}
\usepackage{times}
\usepackage{booktabs}
\usepackage{pifont}
\usepackage{floatrow}
\usepackage{caption}
\usepackage{mathrsfs}
\usepackage[fleqn]{amsmath}
\usepackage{amsfonts,amsthm,amssymb,mathrsfs,bbding}
\usepackage{txfonts}
\usepackage{graphics,multicol}
\usepackage{graphicx}
\usepackage{color}
\usepackage{enumerate}
\usepackage{caption}
\usepackage{cite}
\usepackage{latexsym,bm}
\usepackage{indentfirst}
\usepackage{mathtools}
\evensidemargin=\oddsidemargin\topmargin=-1.5cm

\newtheorem{thm}{Theorem}[section]

\newtheorem{lem}{Lemma}[section]
\newtheorem{cor}{Corollary}[section]

\newtheorem{remark}{Remark}

\theoremstyle{definition}

\addtocounter{section}{0}

\begin{document}

\title{The full chracterazation of the graphs with a L-eigenvalue of multiplicity $n-3$
 \footnote{Supported by the National Natural Science Foundation of China (Grant Nos. 11671344, 11531011, 11626205).}}
\author{{\small Daijun Yin, \ \ Qiongxiang Huang\footnote{
Corresponding author.
Email: yindjmath@163.com, huangqx@xju.edu.cn}}\\[2mm]\scriptsize
College of Mathematics and Systems Science,
\scriptsize Xinjiang University, Urumqi, Xinjiang 830046, P.R. China}
\date{}
\maketitle {\flushleft\large\bf Abstract } Let $\mathcal{G}(n,k)$ be the set of connected graphs of order $n$ with one of the Laplacian eigenvalue having multiplicity $k$. It is well known that $\mathcal{G}(n,n-1)=\{K_n\}$. The graphs of $\mathcal{G}(n,n-2)$ are determined by Das, and
the graphs of $\mathcal{G}(n,n-3)$ with four distinct Laplacian eigenvalues are determined by Mohammadian et al.
In this paper, we determine the graphs of $\mathcal{G}(n,n-3)$ with three distinct Laplacian eigenvalues, and then the full characterization of the graphs in $\mathcal{G}(n,n-3)$ is completed.

\vspace{0.1cm}
\begin{flushleft}
\end{flushleft}

\section{Introduction}
Let $G$ be a undirected simple graph on $n$ vertices with vertex set $V=\{v_1,v_2,\ldots,v_n\}$. The \emph{adjacency matrix} of $G$ is an $n\times n$ matrix $A(G)$ whose $(i,j)$-entry is $1$ if $v_i$ is adjacent to $v_j$ and is $0$ otherwise. The \emph{degree} of $v_i$, denoted by $d_i$, is the number of edges that incident to $v_i$. The matrix $L(G)=D(G)-A(G)$ is called the \emph{Laplacian matrix} of $G$, where $D(G)$ is the $n\times n$ diagonal matrix
whose $(i,i)$-entry is $d_i$. The eigenvalues of $L(G)$ are called \emph{Laplacian eigenvalues} of $G$( short for L-eigenvalues ).
Since $L(G)$ is a positive semidefinite matrix, the L-eigenvalues of $G$ are non-negative and the smallest eigenvalue equals zero.
All the L-eigenvalues together with their multiplicities are called the \emph{Laplacian spectrum} of $G$ ( short for L-spectrum ), denoted by $\emph{Spec}_L(G)$. A graph $G$ is called DLS if $H\cong G$ for any graph $H$ with $\emph{Spec}_L(H)=\emph{Spec}_L(G)$.
Throughout this paper, we denote the multiplicity of L-eigenvalue $\mu$ by $m(\mu)$, and the \emph{diameter} of $G$ by $d(G)$.

Since the connected graph with a few of distinct eigenvalues possess nice combinatorial properties, it arouses a lot of interests for several matrices,
such as the adjacency matrix \cite{Dam3,Dam1,Dam2,Huang}, the Laplacian matrix \cite{Dam,Mohammadian,12},
the signless Laplacian matrix \cite{F} and normalized Laplacian matrix \cite{Dam4}.
We denote by $\mathcal{G}(n,k)$ the set of connected graphs of order $n$ with one of the L-eigenvalue having multiplicity $k$.
It is well known that $\mathcal{G}(n,n-1)=\{K_n\}$. Das \cite{Das} proved that $\mathcal{G}(n,n-2)=\{K_{\frac{n}{2},\frac{n}{2}},K_{1,n-1},K_n-e\}$,
and Mohammadian and Tayfeh-Rezaie \cite{Mohammadian} gave a partial characterization for the graphs of $\mathcal{G}(n,n-3)$.
Motivated by their work, we will complete the characterization of the remaining graphs in $\mathcal{G}(n,n-3)$. By the way, we show that all these graphs of $\mathcal{G}(n,n-3)$ are DLS.

\section{Preliminaries}
The following result is given by Godsil.
\begin{thm}[\cite{God}]\label{2-thm-1}Let $Q=\begin{pmatrix}I_m\mid O\end{pmatrix}^T$ be a $n\times m$ matrix, and let $A$ be a $n\times n$ real symmetric matrix with eigenvalues $\lambda_1\ge\lambda_2\ge\cdots\ge\lambda_n$. If the eigenvalues of $B=Q^TAQ$ are $\mu_1\geq\mu_2\geq\cdots\geq\mu_m$,
then $\lambda_i\ge\mu_i\ge\lambda_{n-m+i}$ ($i=1,\cdots,m$).
Furthermore, if $\mu_i=\lambda_i$ for some $i\in[1,m]$, then $M$ has a $\mu_i$-eigenvector $z$
such that $Qz$ is a $\lambda_i$-eigenvector of $A$.
\end{thm}
There are many pretty properties about Laplacian eigenvalues.
\begin{lem}[\cite{Cvet}]\label{2-lem-1}
Let $G$ be a graph on $n$ vertices with Laplacian eigenvalues $\mu_1\ge\mu_2\ge\cdots\ge\mu_{n-1}\ge\mu_n=0$. Then we have the following results.\\
(i) Denote by $m(0)$ the multiplicity of $0$ as a Laplacian eigenvalue and $w(G)$ the number of connected components of $G$. Then $w(G)=m(0)$.\\
(ii) $G$ has exactly two distinct Laplacian eigenvalues if and only if $G$ is a union of some complete graphs of the same order and some isolate vertices.\\
(iii) The Laplacian eigenvalues of $\bar{G}$ are given by $\mu_i(\bar{G})=n-\mu_{n-i}$ for $i=1,2,\ldots,n-1$ and $\mu_n(\bar{G})=0$.\\
(iv) Let $H$ be a graph on $m$ vertices with Laplacian eigenvalues $\mu_1'\ge\mu_2'\ge\cdots\ge\mu_m'=0$, then the Laplacian spectrum of $G\nabla H$ is \[\{n+m,m+\mu_1,m+\mu_2,\ldots,m+\mu_{n-1},n+\mu_1',n+\mu_2',\ldots,n+\mu_{m-1}',0\}.\]
\end{lem}
It is well-known that
\begin{lem}[\cite{Cvet}]\label{2-lem-2} Let $G$ be a connected graph on $n\geq3$ vertices with $s$ distinct Laplacian eigenvalues.
Then $d(G)\leq s-1$.
\end{lem}

A graph $G$ is said to be a \emph{cograph} if it contains no induced $P_4$. There's a pretty result about cographs.
\begin{lem}[\cite{Corneil}]\label{2-lem-3}
Given a graph $G$, the following statements are equivalent:\\
1) $G$ is a cograph.\\
2) The complement of any connected subgraph of $G$ with at least two vertices is
disconnected.\\
3) Every connected induced subgraph of $G$ has diameter less than or equal to $2$.
\end{lem}

\section{Characterization of $\mathcal{G}(n,n-3)$}

Recall that $\mathcal{G}(n,k)$ is the set of connected graphs  with $m(\mu)=k$ for some L-eigenvalue $\mu$.
If $k=n-1$, then $G$ has exactly two distinct L-eigenvalues, and so $\mathcal{G}(n,n-1)=\{K_n\}$ by Lemma \ref{2-lem-1} (ii).
If $k=n-2$, then $G$ has exactly three distinct L-eigenvalues, and
Das \cite{Das} proved that $\mathcal{G}(n,n-2)=\{K_{\frac{n}{2},\frac{n}{2}},K_{1,n-1},K_n-e\}$.
For a graph $G\in\mathcal{G}(n,n-3)$, we see that $G$ has three or four distinct L-eigenvalues. In terms of the number and multiplicity of Laplacian eigenvalues, we can divide $\mathcal{G}(n,n-3)$ into five classes:
$$\left\{
   \begin{array}{ll}
     \mathcal{G}_1(n,n-3)=\{G\in\mathcal{G}(n,n-3)|\emph{Spec}_L(G)=[(\alpha)^{n-3},(\beta)^2,0]\} \\
     \mathcal{G}_2(n,n-3)=\{G\in\mathcal{G}(n,n-3)|\emph{Spec}_L(G)=[(\alpha)^2,(\beta)^{n-3},0]\} \\
     \mathcal{G}_3(n,n-3)=\{G\in\mathcal{G}(n,n-3)|\emph{Spec}_L(G)=[(\alpha)^{n-3},\beta,\gamma,0]\} \\
     \mathcal{G}_4(n,n-3)=\{G\in\mathcal{G}(n,n-3)|\emph{Spec}_L(G)=[\alpha,(\beta)^{n-3},\gamma,0]\} \\
     \mathcal{G}_5(n,n-3)=\{G\in\mathcal{G}(n,n-3)|\emph{Spec}_L(G)=[\alpha,\beta,(\gamma)^{n-3},0]\}
   \end{array}
 \right.
$$
Mohammadian and Tayfeh-Rezaie \cite{Mohammadian} determine the classes of graphs in $\mathcal{G}_3(n,n-3)$, $\mathcal{G}_4(n,n-3)$ and $\mathcal{G}_5(n,n-3)$, all of which we collect in the following theorem.
\begin{thm}[\cite{Mohammadian}, Theorem 8,9,10]\label{3-thm-1} For an integer $n\geq5$, we have\\
(i) $\mathcal{G}_3(n,n-3)=\{K_{n-3}\nabla\overline{K}_{1,2}\}$ and the Laplacian spectra of it is
\begin{equation}\label{eq-1}
\emph{Spec}_{L}(K_{n-3}\nabla\overline{K}_{1,2})=\{[n]^{n-3},[n-1]^1,[n-3]^1,[0]^1\}.
\end{equation}
(ii) $\mathcal{G}_4(n,n-3)=\{K_1\nabla2K_{\frac{n-1}{2}},\overline{K}_{\frac{n}{3}}\nabla2K_{\frac{n}{3}},K_{n-1}+e,G_{n,r}\}$ where $G_{n,r}$ is obtained from two copies of $K_r\nabla (n-r)K_1$ by joining the vertices of two independent sets $(n-r)K_1$ (shown in Fig. \ref{fig-1}). Furthermore, the Laplacian spectra of them are given by
\begin{equation}\label{eq-2}
\left\{\begin{array}{l}
\emph{Spec}_{L}(K_1\nabla2K_{\frac{n-1}{2}})=\{[n]^1,[\frac{n+1}{2}]^{n-3},[1]^1,[0]^1\}\\
\emph{Spec}_{L}(\overline{K}_{\frac{n}{3}}\nabla2K_{\frac{n}{3}})=\{[n]^1,[\frac{2n}{3}]^{n-3},[\frac{n}{3}]^1,[0]^1\}\\
\emph{Spec}_{L}(K_{n-1}+e)=\{[n]^1,[n-1]^{n-3},[1]^1,[0]^1\}\\
\emph{Spec}_{L}(G_{n,r})=\{[\frac{1}{2}(3n-2r\pm\sqrt{n^2+4nr-4r^2})]^1,[n]^{2n-3},[0]^1\}\\
\end{array}\right..
\end{equation}
(iii) $\mathcal{G}_5(n,n-3)=\{K_{2,n-2},K_{\frac{n}{2},\frac{n}{2}}+e,K_{1,n-1}+e\}$ and the Laplacian spectra of them are given by
\begin{equation}\label{eq-3}
\left\{\begin{array}{l}
\emph{Spec}_{L}(K_{2,n-2})=\{[n]^1,[n-2]^1,[2]^{n-3},[0]^1\}\\
\emph{Spec}_{L}(K_{\frac{n}{2},\frac{n}{2}}+e)=\{[n]^1,[\frac{n}{2}+2]^1,[\frac{n}{2}]^{n-3},[0]^1\}\\
\emph{Spec}_{L}(K_{1,n-1}+e)=\{[n]^1,[3]^1,[1]^{n-3},[0]^1\}\\
\end{array}\right..
\end{equation}
\end{thm}

\begin{figure}
  \centering
\unitlength 3.6mm 
\linethickness{0.4pt}
\ifx\plotpoint\undefined\newsavebox{\plotpoint}\fi 
\begin{picture}(16,9)(0,0)
\put(1,7){\line(1,0){5}}
\put(1,7){\line(5,-2){5}}
\put(1,7){\line(5,-6){5}}
\put(1,5){\line(5,2){5}}
\put(1,5){\line(1,0){4}}
\put(5,5){\line(1,0){1}}
\put(6,5){\line(-5,-4){5}}
\put(1,1){\line(1,0){5}}
\put(6,7){\line(-5,-6){5}}
\put(1,5){\line(5,-4){5}}
\put(6,7){\line(1,0){4}}
\multiput(10,7)(-.06666667,-.03333333){60}{\line(-1,0){.06666667}}
\put(6,5){\line(1,-1){4}}
\put(10,1){\line(-1,0){4}}
\put(6,1){\line(1,1){4}}
\multiput(10,5)(-.06666667,.03333333){60}{\line(-1,0){.06666667}}
\put(10,7){\line(-2,-3){4}}
\put(6,7){\line(2,-3){4}}
\put(6,5){\line(1,0){4}}
\put(10,7){\line(1,0){5}}
\put(10,7){\line(5,-2){5}}
\put(10,7){\line(5,-6){5}}
\put(10,5){\line(5,2){5}}
\put(10,5){\line(1,0){4}}
\put(14,5){\line(1,0){1}}
\put(15,5){\line(-5,-4){5}}
\put(10,1){\line(1,0){5}}
\put(15,7){\line(-5,-6){5}}
\put(10,5){\line(5,-4){5}}
\multiput(.93,4.93)(0,-.8){6}{{\rule{.4pt}{.4pt}}}
\multiput(5.93,4.93)(0,-.75){5}{{\rule{.4pt}{.4pt}}}
\multiput(9.93,4.93)(0,-.8){6}{{\rule{.4pt}{.4pt}}}
\multiput(14.93,4.93)(0,-.8){6}{{\rule{.4pt}{.4pt}}}
\put(1,4){\oval(2,8)[]}
\put(6,4){\oval(2,8)[]}
\put(10,4){\oval(2,8)[]}
\put(15,4){\oval(2,8)[]}
\put(1,9){\makebox(0,0)[cc]{\small $K_r$}}
\put(6.3,9){\makebox(0,0)[cc]{\small $(n-r)K_1$}}
\put(10.5,9){\makebox(0,0)[cc]{\small $(n-r)K_1$}}
\put(15,9){\makebox(0,0)[cc]{\small $K_r$}}
\put(1,7){\circle*{.5}}
\put(1,5){\circle*{.5}}
\put(1,1){\circle*{.5}}
\put(6,7){\circle*{.5}}
\put(6,5){\circle*{.5}}
\put(6,1){\circle*{.5}}
\put(10,7){\circle*{.5}}
\put(10,5){\circle*{.5}}
\put(10,1){\circle*{.5}}
\put(15,7){\circle*{.5}}
\put(15,5){\circle*{.5}}
\put(15,1){\circle*{.5}}
\end{picture}
\caption{\footnotesize The graph $G_{n,r}$.}
\label{fig-1}
\end{figure}
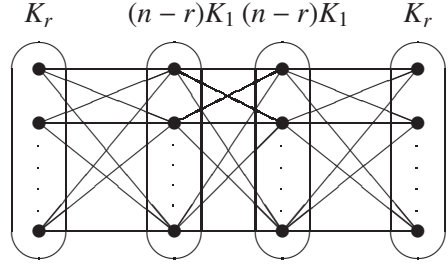

For the complete characterization of $\mathcal{G}(n,n-3)$, it remains to determine the graphs of $\mathcal{G}_1(n,n-3)$ and $\mathcal{G}_2(n,n-3)$.
At first, we introduce a tool which will be used frequently.

\begin{lem}\label{3-lem-1} Let $A$ be a real symmetric matrix of order $n\geq6$ having an eigenvalue $\alpha\neq0$ with $m(\alpha)=n-m$, where $1\leq m\leq n-2$. Let $M$ be a principal submatrix of $A$ of order $m+2$.
Then $\alpha$ is also an eigenvalue of $M$ with $m(\alpha)\ge2$, and there exists eigenvector $z=(z_1,z_2,\ldots,z_{m+2})^T$ of $M$ with respect to $\alpha$ such that $z_k=0$ for $1\le k\le m+2$.
Furthermore, $\sum_{i=1}^{m+2}z_i=0$ if $A=L(G)$ for a graph $G$.
\end{lem}
\begin{proof} Let $\lambda_1\geq\lambda_2\geq\cdots\geq\lambda_n$ be the eigenvalues of $A$ and $\mu_1\ge\mu_2\ge\cdots\ge\mu_{m+2}$ the eigenvalues of $M$.
By our assumption, there exists some $1\le i\le n$ such that $\lambda_i=\cdots=\lambda_{i+n-m-1}=\alpha\neq0$.
By Theorem \ref{2-thm-1}, we get that $\alpha=\lambda_{i+n-m-2}\le\mu_i\le\lambda_i=\alpha$
and $\alpha=\lambda_{i+n-m-1}\le\mu_{i+1}\le\lambda_{i+1}=\alpha$.
Therefore, we have $\mu_i=\mu_{i+1}=\alpha$.
Suppose that $x=(x_1,\ldots,x_{m+2})^T$ and $y=(y_1,\ldots,y_{m+2})^T$ are two independent eigenvectors of $M$ with respect to $\alpha$.
For each fixed integer $1\le k\le m+2$, by linear combination of $x$ and $y$, we get the eigenvector $z=(z_1,\ldots,z_{m+2})^T$ satisfying $z_k=0$.
Furthermore, by Theorem \ref{2-thm-1}, we know that $z*=Qz$ is an eigenvector of $A$ with respect to $\alpha$, where $Q=\begin{pmatrix}I_{m+2}\mid O\end{pmatrix}^T$. Note that the all-ones vector $j$ is an eigenvector of $L(G)$ with respect to $0$, then we have ${z*}^Tj=\sum_{i=1}^{m+2}z_i=0$.
\end{proof}

\begin{lem}\label{3-lem-2} Let $G$ be a connected graph on $n\geq 2m$ vertices with $m(\alpha)=n-m$, where $\alpha$ is a L-eigenvalue of $G$. Then $\alpha$ is integral.
\end{lem}
\begin{proof}
Let $f(x)$ be the characteristic polynomial of $L(G)$. As $L(G)$ only contains integral entries, we obtain that $f(x)$ is a monic polynomial with integral coefficients. Let $p(x)$ be the minimal polynomial of $\alpha$, then $p(x)\in Z[x]$ is irreducible in $Q[x]$ and $(p(x))^{n-m}|f(x)$. We assume that $p(x)$ is a polynomial of degree at least $2$. Therefore, $p(x)$ has another root $\beta\ne0$, which is also a Laplacian eigenvalue of $G$ with multiplicity $n-m$. It follows that $2(n-m)\le n-1$, which implies $n\leq 2m-1$, a contradiction. Thus, we have $p(x)=x-\alpha$ and the result follows.
\end{proof}

\begin{lem}\label{3-lem-3} Let $G\in\mathcal{G}(n,n-3)$ with $n\geq6$, then none of $J_1(=C_5)$, $J_2$ or $J_3$ (shown in Fig. \ref{fig-2}) can be an induced subgraph of $G$.
\end{lem}
\begin{proof} By way of contradiction. Suppose that $G$ has $J_i$ as an induced subgraph and
$N_i$ the principal submatrix of $L(G)$ corresponding to $J_i$, for $i=1,2,3$.
Let $\alpha$ be the Laplacian eigenvalue of $G$ with multiplicity $n-3$.
Then, by Lemma \ref{3-lem-1}, $\alpha$ is an eigenvalue of $N_i$ with $m(\alpha)\geq2$,
and there exists an eigenvector $z$ of $N_i$ with respect to $\alpha$ such that $z_k=0$ for $k\in\{1,\cdots,5\}$, and $\sum_{i=1}^5z_i=0$.

First, suppose that $J_1=C_5$ is an induced subgraph of $G$ and $N_1$ is given by 
\[N_1=\left(\begin{array}{ccccc}d_1&-1&0&0&-1\\-1&d_2&-1&0&0\\0&-1&d_3&-1&0\\0&0&-1&d_4&-1\\-1&0&0&-1&d_5\end{array}\right)\begin{array}{l}v_1\\v_2\\v_3\\v_4\\u\end{array}.\]
By Lemma \ref{3-lem-1}, there exists an eigenvector $x=(0,x_2,x_3,x_4,x_5)^T$ satisfying $x_2+x_3+x_4+x_5=0$ such that $N_1x=\alpha x$.
From the first entry of $N_1x=\alpha x$, we have $x_2+x_5=0$. Therefore, we have $x_3+x_4=0$. If one of $x_2$ and $x_5$ equals zero, then $x_2=x_5=0$. Considering the second entry of both sides of $N_1x=\alpha x$, we get that $x_3=0$, and so $x_4=0$. It leads to that $x=0$, a contradiction. If one of $x_3$ and $x_4$ equals zero, then $x_3=x_4=0$, consider the third entry of both sides of $N_1x=\alpha x$ and we get that $x_2=0$, and so $x_5=0$. It leads to that $x=0$, a contradiction.
Thus $x_2,x_3,x_4,x_5\ne 0$. Without loss of generality, we may assume that $x=(0,1,a,-a,-1)^T$.
Consider the fourth and the fifth entries of both sides of $N_1x=\alpha x$, we have
\[d_5-a=\alpha=d_4+1-\frac{1}{a}.\]
Since $m(\alpha)=n-3$ and $n\geq6$, we get that $\alpha$ is integral by Lemma \ref{3-lem-2}. Then $a$ and $\frac{1}{a}$ are both integral.
Thus, we have $a=\pm1$, and so $\alpha=d_5-1=d_4$ when $a=1$ or $\alpha=d_5+1=d_4+2$ when $a=-1$. It follows that
\begin{equation}\label{eq-5}d_4=d_5-1.\end{equation}
On the other hand, by Lemma\ref{3-lem-1}, there also exists an eigenvector $y=(y_1,0,y_3,y_4,y_5)^T$ satisfying $y_1+y_3+y_4+y_5=0$ such that
$N_1y=\alpha y$. From the second entry of $N_1y=\alpha y$, we have $y_1+y_3=0$. Therefore, we have $y_4+y_5=0$. If one of $y_1$ and $y_3$ equals zero, then $y_1=y_3=0$. Consider the first entry of both sides of $N_1y=\alpha y$, we get that $y_5=0$, and then $y_4=0$.
It leads to that $y=0$, a contradiction.
If one of $y_4$ and $y_5$ equals zero, then $y_4=y_5=0$.
Consider the fourth entry of both sides of $N_1y=\alpha y$, we get that $y_3=0$, and then $y_1=0$.
It leads to that $y=0$, a contradiction. Thus, $y_1,y_3,y_4,y_5\ne 0$. Without loss of generality, we may assume that $y=(b,0,-b,1,-1)^T$.
Consider the fourth and the fifth entries of both sides of $N_1y=\alpha y$, we have
\[d_4+b+1=\alpha=d_5+b+1.\]
It follows that $d_4=d_5$, which contradicts to (\ref{eq-5}).

Second, assume that $J_2$ is an induced subgraph of $G$ and $N_2$ is given by
\[N_2=\left(\begin{array}{ccccc}d_1&-1&0&0&-1\\-1&d_2&-1&0&-1\\0&-1&d_3&-1&0\\0&0&-1&d_4&-1\\-1&-1&0&-1&d_5\end{array}\right)
\begin{array}{l}v_1\\v_2\\v_3\\v_4\\u\end{array}.\]
By Lemma \ref{3-lem-1}, there exists an eigenvector $x=(x_1,x_2,x_3,x_4,0)^T$ satisfying $x_1+x_2+x_3+x_4=0$ such that $N_2x=\alpha x$.
From the fifth entry of $N_2x=\alpha x$, we have $x_1+x_2+x_4=0$, which implies that $x_3=0$.
Then from the third entry of $N_2x=\alpha x$, we get $x_2+x_4=0$, which implies that $x_1=0$,
and then from the first entry we have $x_2=0$. Thus, $x_4=0$ and hence $x=0$, a contradiction.

Finally, assume that $J_3$ is an induced subgraph of $G$ and $N_3$ is given by
\[N_3=\left(\begin{array}{ccccc}d_1&-1&0&0&-1\\-1&d_2&-1&0&-1\\0&-1&d_3&-1&-1\\0&0&-1&d_4&-1\\-1&-1&-1&-1&d_5\end{array}\right)
\begin{array}{l}v_1\\v_2\\v_3\\v_4\\u\end{array}.\]
By Lemma \ref{3-lem-1}, there exists an eigenvector $x=(x_1,x_2,0,x_4,x_5)^T$ satisfying $x_1+x_2+x_4+x_5=0$ such that $N_3x=\alpha x$.
Consider the third, the first and the fourth entries of both sides of $N_3x=\alpha x$ succesively, we get that $x=0$, a contradiction.
\end{proof}

\begin{figure}[htbp]
\begin{center}
\unitlength 3.6mm 
\linethickness{0.4pt}
\ifx\plotpoint\undefined\newsavebox{\plotpoint}\fi 
\begin{picture}(24,9)(0,0)
\thicklines
\put(2,8){\line(1,0){4}}
\multiput(6,8)(.0333333,-.1){30}{\line(0,-1){.1}}
\multiput(7,5)(-.03370787,-.03370787){89}{\line(0,-1){.03370787}}
\multiput(4,2)(-.03370787,.03370787){89}{\line(0,1){.03370787}}
\multiput(1,5)(.0333333,.1){30}{\line(0,1){.1}}
\multiput(9,5)(.0333333,.1){30}{\line(0,1){.1}}
\put(10,8){\line(1,0){4}}
\multiput(14,8)(.0333333,-.1){30}{\line(0,-1){.1}}
\multiput(15,5)(-.03370787,-.03370787){89}{\line(0,-1){.03370787}}
\multiput(12,2)(-.03370787,.03370787){89}{\line(0,1){.03370787}}
\multiput(17,5)(.0333333,.1){30}{\line(0,1){.1}}
\put(18,8){\line(1,0){4}}
\multiput(22,8)(.0333333,-.1){30}{\line(0,-1){.1}}
\multiput(23,5)(-.03370787,-.03370787){89}{\line(0,-1){.03370787}}
\multiput(20,2)(-.03370787,.03370787){89}{\line(0,1){.03370787}}
\put(12,2){\line(-1,3){2}}
\put(18,8){\line(1,-3){2}}
\put(20,2){\line(1,3){2}}
\put(2,8){\circle*{.5}}
\put(6,8){\circle*{.5}}
\put(10,8){\circle*{.5}}
\put(14,8){\circle*{.5}}
\put(18,8){\circle*{.5}}
\put(22,8){\circle*{.5}}
\put(1,5){\circle*{.5}}
\put(7,5){\circle*{.5}}
\put(9,5){\circle*{.5}}
\put(15,5){\circle*{.5}}
\put(17,5){\circle*{.5}}
\put(23,5){\circle*{.5}}
\put(20,2){\circle*{.5}}
\put(12,2){\circle*{.5}}
\put(4,2){\circle*{.5}}
\put(0.3,5){\makebox(0,0)[cc]{\footnotesize$v_1$}}
\put(2,8.7){\makebox(0,0)[cc]{\footnotesize$v_2$}}
\put(6,8.7){\makebox(0,0)[cc]{\footnotesize$v_3$}}
\put(7.7,5){\makebox(0,0)[cc]{\footnotesize$v_4$}}
\put(10,5){\makebox(0,0)[cc]{\footnotesize$v_1$}}
\put(10,8.7){\makebox(0,0)[cc]{\footnotesize$v_2$}}
\put(14,8.7){\makebox(0,0)[cc]{\footnotesize$v_3$}}
\put(15.7,5){\makebox(0,0)[cc]{\footnotesize$v_4$}}
\put(18,5){\makebox(0,0)[cc]{\footnotesize$v_1$}}
\put(18,8.7){\makebox(0,0)[cc]{\footnotesize$v_2$}}
\put(22,8.7){\makebox(0,0)[cc]{\footnotesize$v_3$}}
\put(23.7,5){\makebox(0,0)[cc]{\footnotesize$v_4$}}
\put(4.6,2){\makebox(0,0)[cc]{\footnotesize$u$}}
\put(12.6,2){\makebox(0,0)[cc]{\footnotesize$u$}}
\put(20.6,2){\makebox(0,0)[cc]{\footnotesize$u$}}
\put(4,0.5){\makebox(0,0)[cc]{\small$J_1=C_5$}}
\put(12,0.5){\makebox(0,0)[cc]{\small$J_2$}}
\put(20,0.5){\makebox(0,0)[cc]{\small$J_3$}}
\end{picture}
\caption{The graphs in Lemma \ref{3-lem-3}}
\label{fig-2}
\end{center}
\end{figure}
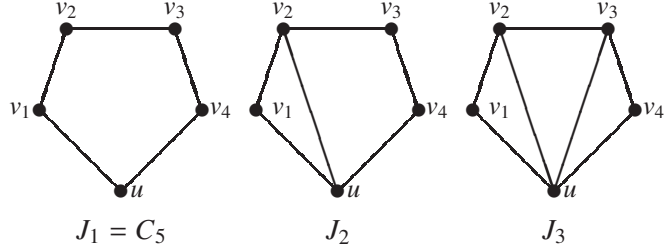

\begin{lem}\label{3-lem-4} Let $G\in\mathcal{G}(n,n-3)$ with $n\ge 6$.
If $G\neq G_{n,r}$ (shown in Fig. \ref{fig-1}), then $\bar{G}$ is disconnected.
\end{lem}
\begin{proof} By Lemma \ref{2-lem-3}, it suffices to show that $G$ contains no induced $P_4$. Assume by contradiction that $G$ contains an induced $P_4=v_1v_2v_3v_4$. 
If $G$ has three distinct Laplacian eigenvalues, then $d(G)=2$ by Lemma \ref{2-lem-2}. Therefore, there exists a vertex $u\in V(G)$ such that $u\sim v_1,v_4$, since otherwise $d(v_1,v_4)\ge3$. It follows that at least one of $J_1$, $J_2$ and $J_3$ will be an induced subgraph of $G$, which contradicts to Lemma \ref{3-lem-3}.
If $G$ has four distinct Laplacian eigenvalues, then $G\in\{K_1\nabla2K_{\frac{n-1}{2}},\overline{K}_{\frac{n}{3}}\nabla2K_{\frac{n}{3}},K_{n-1}+e,G_{n,r}\}$ by Theorem \ref{3-thm-1}(ii), from which we find that all graphs but $G_{n,r}$ have their complements disconnected.
\end{proof}

We now give the characterization of the graphs in $\mathcal{G}_1(n,n-3)$ and $\mathcal{G}_2(n,n-3)$.
\begin{thm}\label{3-thm-3} For an integer $n\geq6$, we have \\
(i) $\mathcal{G}_1(n,n-3)=\{3K_1\nabla K_{n-3},C_4\nabla K_{n-4}\}$ and the Laplacian spectra of them are given by
\begin{equation}\label{eq-7}
\left\{\begin{array}{l}
\emph{Spec}_{L}(3K_1\nabla K_{n-3})=\{[n]^{n-3},[n-3]^2,[0]^1\}\\
\emph{Spec}_{L}(C_4\nabla K_{n-4})=\{[n]^{n-3},[n-2]^2,[0]^1\}\\
\end{array}\right..
\end{equation}
(ii) $\mathcal{G}_2(n,n-3)=\{K_2\nabla(n-2)K_1,K_1\nabla K_{\frac{n-1}{2},\frac{n-1}{2}},K_{\frac{n}{3},\frac{n}{3},\frac{n}{3}}\}$ and
the Laplacian spectra of them are given by
\begin{equation}\label{eq-8}
\left\{\begin{array}{l}
\emph{Spec}_{L}(K_2\nabla(n-2)K_1)=\{[n]^2,[2]^{n-3},[0]^1\}\\
\emph{Spec}_{L}(K_1\nabla K_{\frac{n-1}{2},\frac{n-1}{2}})=\{[n]^2,[\frac{n+1}{2}]^{n-3},[0]^1\}\\
\emph{Spec}_{L}(K_{\frac{n}{3},\frac{n}{3},\frac{n}{3}})=\{[n]^2,[\frac{2n}{3}]^{n-3},[0]^1\}\\
\end{array}\right.
\end{equation}
\end{thm}
\begin{proof} Let $G\in\mathcal{G}_1(n,n-3)\cup\mathcal{G}_2(n,n-3)$. Then $G$ has three distinct Laplacian eigenvalues, say $\alpha>\beta>0$, and so $G\neq G_{n,r}$. By Lemma \ref{3-lem-4}, we know that $\bar{G}$ is disconnected,
and so $\alpha=n$ from Lemma \ref{2-lem-1} (i) and (iii).

First suppose that $\alpha=n$ has multiplicity $n-3$, and so $G$ has Laplacian spectrum $\{n^{n-3},\beta^2,0\}$.
Then, by Lemma \ref{2-lem-1}, the L-spectrum of $\bar{G}$ is given by
\[\mathrm{Spec}_L(\bar{G})=\{[n-\beta]^2,[0]^{n-2}\},\]
which implies that $\bar{G}$ has $n-2$ connected components, say $\bar{G}=G_1\cup G_2\cup\cdots\cup G_{n-2}$.
It is easy to see that $\bar{G}=K_3\cup(n-3)K_1$ or $2K_{2}\cup(n-4)K_1$. Consequently, $G=3K_1\nabla K_{n-3}$ or $G=C_4\nabla K_{n-4}$.
Thus (i) holds.

Next suppose that $\beta$ has multiplicity $n-3$, and so $G$ has Laplacian spectra $\{n^2,\beta^{n-3},0\}$.
Similarly, by Lemma \ref{2-lem-1}, we have
\[Spec_L(\bar{G})=\{[n-\beta]^{n-3},[0]^3\},\]
which implies that $\bar{G}$ contains three connected components, say $\bar{G}=G_1\cup G_2\cup G_3$.
It is routine to verify that $\bar{G}=2K_1\cup K_{n-2}$, $K_1\cup 2K_{\frac{n-1}{2}}$ or $3K_{\frac{n}{3}}$. Hence $G=K_2\nabla (n-2)K_1$, $K_1\nabla K_{\frac{n-1}{2},\frac{n-1}{2}}$ or $K_{\frac{n}{3},\frac{n}{3},\frac{n}{3}}$.
Thus (ii) holds.

Note that all of the graphs we find consist of the join of two graphs, by Lemma \ref{2-lem-1} (iv),
we obtain their Laplacian spectra, which are shown in (\ref{eq-7}) and (\ref{eq-8}).
\end{proof}

\begin{remark}\label{re-1} By using the software SageMath, we get the graphs of $\mathcal{G}(n,n-3)$ with $n=4$ and $n=5$. That is,
\[\left\{\begin{array}{l}
\mathcal{G}(4,1)=\{P_4,K_{1,3}+e\}\\
\mathcal{G}(5,2)=\{C_5,K_1\nabla C_4,K_2\nabla3K_1\}\\
\end{array}\right.\]
\end{remark}

\begin{cor}\label{3-cor-1} Let $G\in\mathcal{G}(n,n-3)$ with $n\geq4$, then $G$ is determined by its Laplacian spectrum.
\end{cor}
\begin{proof} Let $H$ be a connected graph on $n$ vertices with $\emph{Spec}_L(H)=\emph{Spec}_L(G)$. We get that $H\in\mathcal{G}(n,n-3)$.
Then the result follows by pairwise comparing the Laplacian spectra of graphs in $\mathcal{G}(n,n-3)$.
\end{proof}

\end{document}